\documentclass[12pt,reqno]{amsart}
\usepackage{amsmath,amssymb,amsfonts,amsthm,cite}

\newtheorem{thm}{Theorem}[section]
\newtheorem{lem}[thm]{Lemma}
\newtheorem{cor}[thm]{Corollary}
\theoremstyle{remark}
\newtheorem{rem}[thm]{Remark}

\numberwithin{equation}{section}

\begin{document}
	\allowdisplaybreaks
	\title
	{Parameterized partial theta identities and a unified $q$-difference proof}
	
	\author[C.-Y. Su]{Chen-Yang Su}
	\address[C.-Y. Su]{College of Mathematical Science, Tianjin Normal University, Tianjin 300387, P.R. China}
	\email{suchenyang@tjnu.edu.cn}
	
	\date{}
	\begin{abstract}
		We establish five families of integer-parameter extensions of Ramanujan's partial theta identities. The families follow from a common two-parameter specialization of Andrews' transformation, for which we give an independent proof based on a $q$-difference recurrence and a boundary estimate. Specializations of the integer parameter recover six identities from Ramanujan's lost notebook. As an application, a residue argument applied to the fifth family gives an integer-parameter extension of Lovejoy's residual identity, from which we construct a corresponding family of conjugate Bailey pairs.
	\end{abstract}
	
	\maketitle
	\section{Introduction}
	To state the main results, we introduce the standard notation of \cite{Gasper-Rahman-2004}. Let $q$ be a complex number with $0<|q|<1$. For a nonnegative integer $n$,
	\[
	(a;q)_{0}:=1, \quad (a;q)_n:=\prod_{k=0}^{n-1}(1-aq^{k}), \quad (a;q)_{\infty}:=\prod_{k=0}^{\infty}(1-aq^{k}),
	\]
	and
	\[
	(a;q)_{-n}:=\frac{1}{(aq^{-n};q)_n}.
	\]
	For $r\geq1$, we also write
	\[
	(a_1,a_2,\ldots, a_r;q)_n:=(a_1;q)_n(a_2;q)_n\cdots(a_r;q)_n,
	\]
	where $n$ may be an integer or $\infty$ whenever the expression is defined.
	
	A partial theta function is a unilateral series of the form
	\[\sum^{\infty}_{n=0}z^nq^{An^2+Bn},\]
	obtained by restricting the summation in a theta series to nonnegative indices. In \cite[p.~37]{Ramanujan-1988}, Ramanujan recorded several partial theta function identities without proofs. For example, for $x \neq 0$,
	\begingroup
	\setlength{\jot}{2pt}
	\setlength{\abovedisplayskip}{6pt}
	\setlength{\belowdisplayskip}{6pt}
	\setlength{\abovedisplayshortskip}{4pt}
	\setlength{\belowdisplayshortskip}{4pt}
	\begin{align}
		\sum^{\infty}_{n=0}\frac{q^n}{(xq,x^{-1}q;q)_n}
		&=(1-x)\sum_{n=0}^{\infty}(-1)^nx^{3n}q^{n(3n+1)/2}
		(1-x^2q^{2n+1})\nonumber\\*
		&\quad +\frac{x}{(xq,x^{-1}q;q)_{\infty}}
		\sum_{n=0}^{\infty}(-1)^nx^{2n}q^{\binom{n+1}{2}},\label{p-1}\\[5pt]
		\sum^{\infty}_{n=0}\frac{q^{2n+1}}{(-xq,-x^{-1}q;q^2)_{n+1}}
		&=\sum_{n=0}^{\infty}x^{3n+1}q^{n(3n+2)}(1-xq^{2n+1})\nonumber\\
		&\quad -\frac{1}{(-xq,-x^{-1}q;q^2)_{\infty}}
		\sum_{n=0}^{\infty}(-1)^nx^{2n+1}q^{n(n+1)}.\label{p-2}
	\end{align}
	\begin{align}
		\sum^{\infty}_{n=0}\frac{(q;q^2)_nq^{2n}}{(-xq^2,-x^{-1}q^2;q^2)_{n}}
		&=(1+x)\sum_{n=0}^{\infty}(-1)^nx^{n}q^{n(n+1)/2}\nonumber\\
		&\quad -\frac{x(q;q^2)_{\infty}}
		{(-xq^2,-x^{-1}q^2;q^2)_{\infty}}
		\sum_{n=0}^{\infty}x^{3n}q^{3n^2+2n}
		(1-xq^{2n+1}).\label{p-6}
	\end{align}
	\begin{multline}
		\left(1+\frac{1}{x}\right)
		\sum^{\infty}_{n=0}\frac{(q;q^2)_nq^{2n+1}}
		{(-xq,-x^{-1}q;q^2)_{n+1}}
		=\sum_{n=0}^{\infty}(-1)^nx^{n}q^{n(n+1)/2}\\*
		-\frac{(q;q^2)_{\infty}}
		{(-xq,-x^{-1}q;q^2)_{\infty}}
		\sum_{n=0}^{\infty}x^{3n}q^{3n^2+n}
		(1-x^2q^{4n+2}).\label{p-3}
	\end{multline}
	\begin{align}
		\sum^{\infty}_{n=0}\frac{(q;q^2)_nq^{n}}{(-xq,-x^{-1}q;q)_{n}}
		&=(1+x)\sum_{n=0}^{\infty}(-x)^nq^{n(n+1)/2}\nonumber\\
		&\quad -\frac{x(q;q^2)_{\infty}}
		{(-xq,-x^{-1}q;q)_{\infty}}
		\sum_{n=0}^{\infty}(-1)^nx^{2n}q^{n^2+n},\label{p-4}\\[5pt]
		\left(1+\frac{1}{x}\right)\sum^{\infty}_{n=0}\frac{(-q;q)_{2n}q^{2n+1}}{(xq,x^{-1}q;q^2)_{n+1}}
		&=-\sum_{n=0}^{\infty}(-1)^nx^{n}q^{n(n+1)}\nonumber\\
		&\quad +\frac{(-q;q)_{\infty}}
		{(xq,x^{-1}q;q^2)_{\infty}}
		\sum_{n=0}^{\infty}(-1)^{n}x^{n}q^{(n^2+n)/2}.\label{p-5}
	\end{align}
	\endgroup
	In 1981, Andrews \cite{Andrews-1981} systematically studied Ramanujan's partial theta function identities and proved them using $q$-summation and transformation formulas. Subsequently, Agarwal \cite{Agarwal-1984} related Andrews' transformation to a ${}_3\phi_2$ identity due to Sears \cite{Sears-1951}.
	
	Further extensions were obtained using Bailey pairs, products of partial theta functions, and transformation methods; see, for example, \cite{Warnaar-2003,Andrews-Warnaar-2007,Berkovich-2020,Ma-2012,Wang-Ma-2018,Sun-2020}. Residual identities and their partition-theoretic consequences were studied in \cite{Lovejoy-2012,Kim-Lovejoy-2015,Kim-Lovejoy-2018}, while combinatorial proofs of related identities were given in \cite{Alladi-2009,Alladi-2010,Berndt-Kim-Yee-2010,Kim-2010}.
	
	We study integer-parameter extensions of these six identities through a common specialization of Andrews' transformation, stated in Lemma~\ref{unified-identity}. We give this identity an independent proof by showing that both sides satisfy a common $q$-difference recurrence and controlling their iterated difference by a boundary estimate. Introducing the integer parameter $m$ produces shifted infinite products together with finite correction terms, and suitable choices of $B$ and $C$ yield Theorems~\ref{partial-thm1}--\ref{partial-thm5}, which recover the six Ramanujan identities \eqref{p-1}--\eqref{p-5} as special cases. Following Lovejoy's residue approach \cite{Lovejoy-2012}, we derive an integer-parameter extension of his residual identity from Theorem~\ref{partial-thm5} and use it to construct an explicit family of conjugate Bailey pairs.
	
	The main results are stated below. Throughout the paper, $m$ is an integer in the range specified in each result, and $\sum_{n=0}^{k}=0$ when $k<0$. All parameters are assumed to be chosen so that no displayed denominator vanishes.
	To streamline the statements, set
	\[
	D_{r,n}(x):=(-xq^r,-x^{-1}q^r;q^2)_n,
	\qquad D_r(x):=D_{r,\infty}(x).
	\]
	
	\begingroup
	\setlength{\jot}{0pt}
	\begin{thm}\label{partial-thm1} For $m\geq 0$ and $x\neq0$,
		\begin{align*}
			&\sum_{n=0}^{\infty}\frac{q^{2n+m}}{D_{m+2,n}(x)}\\
			&=q^m(1+xq^m)\sum_{n=0}^{m-1}
			(-xq^{-m};q^2)_{n+1}(-xq^m)^n+
			\frac{1}{D_{m+2}(x)}
			\sum_{n=0}^{\infty}(-1)^{n+1}x^{2n+1}q^{n^2+n}\\
			&\quad+(-1)^mq^{m^2+m}(-xq^{-m};q^2)_{m+1}\sum_{n=0}^{\infty}x^{3n+m}q^{3n^2+(3m+1)n}
			(1-x^2q^{4n+2m+2}).
		\end{align*}
	\end{thm}

	\begin{thm}\label{partial-thm2} For $m\geq 0$ and $x\neq0$,
		\begin{align*}
			&\sum_{n=0}^{\infty}\frac{(q;q^2)_nq^{2n+m}}{D_{m+2,n}(x)}\\
			&=\frac{(q;q^2)_{\infty}}{D_{m+2}(x)}
			\sum_{n=0}^{m-1}
			\frac{(-1)^{n+1}x^{2n+1}q^{n^2+n}}
			{(-xq^{-m+1};q^2)_{n+1}}+(1+xq^m)q^m\sum_{n=0}^{m-1}
			\frac{(-xq^{-m};q^2)_{n+1}(-xq^m)^n}
			{(-xq^{-m+1};q^2)_{n+1}}\\
			&\quad+\frac{(q;q^2)_{\infty}(-1)^{m+1}x^{2m+1}q^{m^2+m}}
			{D_{m+2}(x)(-xq^{-m+1};q^2)_m}
			\sum_{n=0}^{\infty}x^{3n}q^{3n^2+(3m+2)n}
			(1-xq^{2n+m+1})\\
			&\quad+\frac{(-xq^{-m};q^2)_{m+1}(-x)^mq^{m^2+m}}
			{(-xq^{-m+1};q^2)_m}
			\sum_{n=0}^{\infty}x^{2n}q^{2n^2+(2m+1)n}
			(1-xq^{2n+m+1}).
		\end{align*}
	\end{thm}

	\enlargethispage{2\baselineskip}
	\begin{thm}\label{partial-thm3} For $m\geq0$, $x\neq0$, and $|xq^{1-m}|<1$,
		\begin{align*}
			&\sum_{n=0}^{\infty}\frac{(-1;q)_{2n}q^{2n+m}}{D_{m,n+1}(x)}\\
			&=-\frac{(-1;q)_{\infty}}
			{(xq^{-m};q)_{2m+1}D_m(x)}\sum_{n=0}^{\infty}
			(xq^{2n-m+2};q)_{2m-1}x^{2n+1}
			q^{2n^2+(1-2m)n}(1-x^2q^{4n+2})\\
			&\quad+x\sum_{n=0}^{m-1}
			\frac{(q^{-2m+1},-xq^{-m+2};q^2)_n(-xq^m)^n}
			{(xq^{-m};q)_{2n+2}}\\
			&\quad+\frac{(-xq^{-m+2},q^{-2m+1};q^2)_m(-xq^m)^m}
			{(1+xq^m)(xq^{-m};q)_{2m+1}}\sum_{n=0}^{\infty}
			x^{n+1}q^{n(n+m)}(1+xq^{2n+m+1}).
		\end{align*}
	\end{thm}

	\begin{thm}\label{partial-thm4} For $m\geq0$, $x\neq0$, and $|xq^{1-m}|<1$,
		\begin{align*}
			&\sum_{n=0}^{\infty}\frac{(q^2;q^4)_nq^{2n+m}}{D_{m,n+1}(x)}\\*
			&=-\frac{x(q^2;q^4)_{\infty}}
			{(x^2q^{2-2m};q^4)_mD_m(x)}\sum_{n=0}^{\infty}(-1)^n
			(x^2q^{4n-2m+6};q^4)_{m-1}x^{2n}
			q^{2n^2+(2-2m)n}(1-x^2q^{4n+2})\\
			&\quad+x\sum_{n=0}^{m-1}
			\frac{(-xq^{-m+2},-q^{-2m+2};q^2)_n(-xq^m)^n}
			{(x^2q^{2-2m};q^4)_{n+1}}\\
			&\quad+\frac{(-xq^{-m+2};q^2)_{m-1}
				(-q^{-2m+2};q^2)_m(-xq^m)^m}
			{(x^2q^{2-2m};q^4)_m}\sum_{n=0}^{\infty}
			(-1)^nx^{n+1}q^{n(n+m+1)}.
		\end{align*}
	\end{thm}

	\begin{thm}\label{partial-thm5} For $m \geq 1$, $x\neq0$, and $|xq^{2-m}|<1$,
		\begin{align*}
			&\sum_{n=0}^{\infty}\frac{(-q;q)_{2n}q^{2n+m}}{D_{m,n+1}(x)}\\
			&=-\frac{(-q;q)_{\infty}}
			{(xq^{-m+1};q)_{2m-1}D_m(x)}\sum_{n=0}^{\infty}
			(xq^{2n-m+3};q)_{2m-3}x^{2n+1}
			q^{2n^2+(3-2m)n}(1-x^2q^{4n+2})\\
			&\quad+x\sum_{n=0}^{m-2}
			\frac{(-xq^{-m+2},q^{-2m+3};q^2)_n(-xq^m)^n}
			{(xq^{-m+1};q)_{2n+2}}\\
			&\quad+\frac{(-xq^{-m+2},q^{-2m+3};q^2)_{m-1}
				(-xq^m)^{m-1}}
			{(xq^{-m+1};q)_{2m-1}}
			\sum_{n\geq0}x^{n+1}q^{n(n+m)}.
		\end{align*}
	\end{thm}
	\endgroup
	
	\begin{rem}
		When $m=1$, Theorem~\ref{partial-thm5}, after replacing $x$ by $-x$, reduces to Ramanujan's identity \eqref{p-5}. Kim and Lovejoy \cite[Eq.~(1.9)]{Kim-Lovejoy-2015} used this identity to study rank differences for a class of special unimodal sequences. Thus Theorem~\ref{partial-thm5} may also be viewed as an $m$-parameter extension of the partial theta identity underlying their work.
	\end{rem}
	
	The paper is organized as follows. Section~\ref{sec-preliminaries} collects
	standard identities and proves the common two-parameter identity by a
	$q$-difference recurrence. Section~\ref{sec-proofs} derives
	Theorems~\ref{partial-thm1}--\ref{partial-thm5} from this identity by taking
	\[
	(B,C)=(0,0),\ (q,0),\ (-1,-q),\ (q,-q),\ (-q,-q^2),
	\]
	respectively. Section~\ref{sec-special-cases} recovers the six Ramanujan
	identities displayed in the introduction. Section~\ref{sec-bailey} derives a
	residual identity and the corresponding conjugate Bailey pair from
	Theorem~\ref{partial-thm5}.
	
	\section{Preliminaries and a unified identity}\label{sec-preliminaries}
	Recall the definition
	\begin{align*}
		_r\phi_s\left(\begin{array}{c}
			a_1,a_2,\ldots, a_r\\
			b_1,b_2,\ldots, b_s
		\end{array}
		;q,x\right)=\sum_{n=0}^{\infty}\frac{(a_1,a_2,\ldots,a_r;q)_n}
		{(q,b_1,\ldots,b_s;q)_n}\left((-1)^nq^{\binom{n}{2}}\right)^{1+s-r}x^n.
	\end{align*}
	We shall use the following standard identities.
	
	\begin{lem}\cite[Theorem~1.2.1]{Andrews-Berndt-2009} \label{Heine}
		If $h$ is a positive integer, then, for $|t|,|b|<1$,
		\begin{align}\label{Heineh}
			\sum_{n=0}^{\infty}\frac{(a;q^h)_n(b;q)_{hn}t^n}{(q^h;q^h)_n(c;q)_{hn}}
			=\frac{(b;q)_{\infty}(at;q^h)_{\infty}}{(c;q)_{\infty}(t;q^h)_{\infty}}
			\sum_{n=0}^{\infty}\frac{(c/b;q)_n(t;q^h)_{n}b^n}{(q;q)_n(at;q^h)_n}.
		\end{align}
	\end{lem}
	
	\begin{lem}[Jackson's transformation]\label{jackson-transformation}
		\cite[Appendix (III.4)]{Gasper-Rahman-2004}
		Let $b\neq0$ and $|z|<1$, and suppose that no denominator below
		vanishes. Then
		\begin{align*}
			{}_2\phi_1\left(\begin{array}{c}a,b\\ e\end{array};q^2,z\right)
			=\frac{(az;q^2)_\infty}{(z;q^2)_\infty}
			{}_2\phi_2\left(\begin{array}{c}a,e/b\\ e,az\end{array};q^2,bz\right).
		\end{align*}
	\end{lem}
	
	\begin{lem}[Finite $q$-binomial theorem]\label{finite-q-binomial}
		\cite{Gasper-Rahman-2004}
		For every nonnegative integer $r$ and every complex $a$,
		\begin{align*}
			(a;q^2)_r
			=\sum_{j=0}^{r}(-a)^jq^{j(j-1)}
			\genfrac{[}{]}{0pt}{}{r}{j}_{q^2},\qquad
			\genfrac{[}{]}{0pt}{}{r}{j}_{q^2}
			:=\frac{(q^2;q^2)_r}{(q^2;q^2)_j(q^2;q^2)_{r-j}}.
		\end{align*}
	\end{lem}
	
	\begin{lem}\label{2.11}\cite[Appendix (III.9)]{Gasper-Rahman-2004} For $|de/abc|, |e/a|<1$,
		\begin{align}\label{3phi2}
			_3\phi_2\left(\begin{array}{c}
				a,b,c\\
				d,e
			\end{array}
			;q,\frac{de}{abc}\right)=\frac{(e/a;q)_{\infty}(de/bc;q)_{\infty}}{(e;q)_{\infty}(de/abc;q)_{\infty}} \ _3\phi_2\left(\begin{array}{c}
				a,d/b,d/c\\
				d,de/bc
			\end{array}
			;q,\frac{e}{a}\right).
		\end{align}
	\end{lem}
	
	\begin{lem}\label{lem2}\cite[p. 15]{Fine-1988}
		For $|\tau|<1$, and provided that no denominator vanishes, the
		Rogers-Fine identity is
		\begin{align}\label{RF}
			\sum^{\infty}_{n=0}\frac{(\alpha;q)_n\tau^n}{(\beta;q)_n}
			=\sum^{\infty}_{n=0}\frac{(\alpha;q)_n(\alpha\tau q/\beta;q)_n\beta^n\tau^nq^{n^2-n}(1-\alpha\tau q^{2n})}{(\beta;q)_n(\tau;q)_{n+1}}.
		\end{align}
	\end{lem}
	
	We record explicitly the limiting case of \eqref{RF} that will be used
	in the proof of Theorem~\ref{partial-thm1}.
	
	\begin{lem}\label{RF-zero}
		For $|\tau|<1$,
		\begin{align}
			\sum_{n=0}^{\infty}(\alpha;q^2)_n\tau^n
			&=\sum_{n=0}^{\infty}
			\frac{(\alpha;q^2)_n(-\alpha)^n\tau^{2n}q^{3n^2-n}}
			{(\tau;q^2)_{n+1}}
			(1-\alpha\tau q^{4n}).
			\label{RF-zero-identity}
		\end{align}
	\end{lem}
	
	\begin{proof}
		Replace $q$ by $q^2$ in \eqref{RF} and initially take $\beta\neq0$.
		For every fixed $n$,
		\begin{align}
			(\alpha\tau q^2/\beta;q^2)_n\beta^n
			&=\prod_{k=0}^{n-1}(\beta-\alpha\tau q^{2k+2}),
			\label{RF-zero-product}
		\end{align}
		and hence
		\begin{align*}
			\lim_{\beta\to0}
			(\alpha\tau q^2/\beta;q^2)_n\beta^n
			&=(-\alpha\tau)^nq^{n(n+1)}.
		\end{align*}
		Moreover, this limit may be passed through both series. Indeed, choose
		$0<r<1$. For $|\beta|\leq r$,
		\begin{align*}
			|(\beta;q^2)_n|
			&\geq \prod_{k=0}^{\infty}(1-r|q|^{2k})>0,\\
			|(\alpha\tau q^2/\beta;q^2)_n\beta^n|
			&\leq (r+|\alpha\tau||q|^2)^n.
		\end{align*}
		The remaining finite products are bounded uniformly in $n$, while the
		factor $|q|^{2n^2-2n}$ on the right-hand side of \eqref{RF} gives a
		summable majorant independent of $\beta$. The series on the left-hand
		side is uniformly dominated by a constant multiple of $|\tau|^n$.
		Thus both series converge uniformly for $|\beta|\leq r$, and termwise
		passage to the limit is justified. In particular, after the product in
		\eqref{RF-zero-product} is used to remove the apparent singularity, the
		right-hand side has a continuous extension to $\beta=0$. Finally, using
		$(\beta;q^2)_n\to1$ and
		\[
		n(n+1)+2n^2-2n=3n^2-n
		\]
		gives \eqref{RF-zero-identity}.
	\end{proof}
	
	The following specialized identity is the common starting point for all
	five main theorems.  We give it a direct $q$-difference proof without
	appealing to Andrews' derivation.
	
	\begin{lem}\label{unified-identity}
		Let $m\geq0$ and $x\neq0$.  Suppose that none of the denominators below
		vanishes, $|xq^m|<1$, and, when $C\neq0$, that
		$|Cxq^{-m}|<1$.  Then
		\begin{align}
			&\sum_{n=0}^{\infty}
			\frac{(B,C;q^2)_nq^{2n}}
			{(-xq^{m+2},-x^{-1}q^{m+2};q^2)_n}
			\nonumber\\
			&=\frac{-xq^{-m}(B,C;q^2)_\infty}
			{(-xq^{m+2},-x^{-1}q^{m+2};q^2)_\infty}
			\sum_{n=0}^{\infty}
			\frac{(-xq^{m+2}/C;q^2)_n(-Cxq^{-m})^n}
			{(-Bxq^{-m};q^2)_{n+1}}
			\nonumber\\
			&\quad +(1+xq^m)\sum_{n=0}^{\infty}
			\frac{(-xq^{-m};q^2)_{n+1}(BCq^{-2m};q^2)_n(-xq^m)^n}
			{(-Bxq^{-m},-Cxq^{-m};q^2)_{n+1}}.
			\label{master-identity}
		\end{align}
		For $C=0$, the first series on the right-hand side is understood as
		its limit as $C\to0$.
	\end{lem}
	
	
	\begin{proof}
		We first derive a common recurrence and then iterate the resulting
		error.  The only delicate point is the boundary estimate, where two
		terms of order $q^{-2N}$ cancel.
		
		\smallskip
		\noindent\emph{Step 1: the first two recurrences.}
		Denote the left-hand side of \eqref{master-identity} by
		$\mathcal F_m(B,C)$ and the two terms on its right-hand side by
		$\mathcal U_m(B,C)$ and $\mathcal V_m(B,C)$, respectively.  Put
		$\mathcal R_m(B,C)=\mathcal U_m(B,C)+\mathcal V_m(B,C)$.
		Separating the term $n=0$ and then replacing $n$ by $n+1$ gives
		\begin{align}
			\mathcal F_m(B,C)
			&=1+\lambda_m(B,C)\mathcal F_{m+2}(Bq^2,Cq^2),
			\label{master-F-recurrence}\\
			\lambda_m(B,C)
			&:=\frac{q^2(1-B)(1-C)}
			{(1+xq^{m+2})(1+x^{-1}q^{m+2})}.
			\label{master-lambda}
		\end{align}
		Directly from the infinite products and the first series on the
		right-hand side of \eqref{master-identity}, we also have
		\begin{align}
			\mathcal U_m(B,C)
			=\lambda_m(B,C)\mathcal U_{m+2}(Bq^2,Cq^2).
			\label{master-U-recurrence}
		\end{align}
		
		\smallskip
		\noindent\emph{Step 2: the recurrence for the second term on the
			right-hand side.}
		We next treat $\mathcal V_m(B,C)$.  Set
		\begin{align}
			c=-xq^{-m},\qquad t=-xq^m,\qquad
			d=BCq^{-2m}=\frac{BCc}{t},
			\label{master-ctd}
		\end{align}
		and define
		\begin{align}
			A_n=\frac{(c,d;q^2)_nt^n}{(Bc,Cc;q^2)_n}.
			\label{master-An}
		\end{align}
		From the definition of $A_n$,
		\begin{align*}
			A_{n+1}=A_n
			\frac{t(1-cq^{2n})(1-dq^{2n})}
			{(1-Bcq^{2n})(1-Ccq^{2n})}.
		\end{align*}
		Since $td=BCc$, a direct expansion gives
		\begin{align*}
			&(1-Bcq^{2n})(1-Ccq^{2n})
			-t(1-cq^{2n})(1-dq^{2n})\\
			&\qquad=(1-t)(1-cq^{2n})
			+c(1-B)(1-C)q^{2n}.
		\end{align*}
		Combining these two relations gives
		\begin{align}
			A_n-A_{n+1}
			=(1-t)\frac{(c;q^2)_{n+1}(d;q^2)_nt^n}
			{(Bc,Cc;q^2)_{n+1}}
			+c(1-B)(1-C)
			\frac{(c,d;q^2)_n(tq^2)^n}
			{(Bc,Cc;q^2)_{n+1}}.
			\label{master-telescoping}
		\end{align}
		The sum of the first term on the right-hand side of
		\eqref{master-telescoping} is $\mathcal V_m(B,C)$, since
		\begin{align*}
			\mathcal V_m(B,C)
			=(1-t)\sum_{n=0}^{\infty}
			\frac{(c;q^2)_{n+1}(d;q^2)_nt^n}
			{(Bc,Cc;q^2)_{n+1}}.
		\end{align*}
		For the second term, put
		\begin{align*}
			S_m(B,C):=\sum_{n=0}^{\infty}
			\frac{(c,d;q^2)_n(tq^2)^n}
			{(Bc,Cc;q^2)_{n+1}}.
		\end{align*}
		Upon replacing $m,B,C$ by $m+2,Bq^2,Cq^2$, respectively, the three
		quantities $c,t,d$ defined in \eqref{master-ctd} become
		$cq^{-2},tq^2,d$. We consequently obtain
		\begin{align*}
			\mathcal V_{m+2}(Bq^2,Cq^2)
			=(1-tq^2)(1-cq^{-2})S_m(B,C).
		\end{align*}
		Moreover, $1+xq^{m+2}=1-tq^2$ and
		\[
		1+x^{-1}q^{m+2}=1-\frac{q^2}{c}
		=-\frac{q^2}{c}(1-cq^{-2}).
		\]
		Hence, by \eqref{master-lambda}, we obtain
		\begin{align*}
			c(1-B)(1-C)S_m(B,C)
			=-\lambda_m(B,C)\mathcal V_{m+2}(Bq^2,Cq^2).
		\end{align*}
		Finally, $A_0=1$, while $A_n\to0$ because $|t|<1$ and the quotient
		of the finite products in \eqref{master-An} remains bounded.  Summing
		\eqref{master-telescoping} over $n\geq0$ therefore gives
		\begin{align*}
			1=\mathcal V_m(B,C)
			-\lambda_m(B,C)\mathcal V_{m+2}(Bq^2,Cq^2),
		\end{align*}
		which is equivalent to
		\begin{align}
			\mathcal V_m(B,C)
			=1+\lambda_m(B,C)\mathcal V_{m+2}(Bq^2,Cq^2).
			\label{master-V-recurrence}
		\end{align}
		It follows from \eqref{master-U-recurrence} and
		\eqref{master-V-recurrence} that
		\begin{align}
			\mathcal R_m(B,C)
			=1+\lambda_m(B,C)\mathcal R_{m+2}(Bq^2,Cq^2).
			\label{master-R-recurrence}
		\end{align}
		
		\smallskip
		\noindent\emph{Step 3: iteration of the error.}
		Define
		\[
		\mathcal E_m(B,C):=\mathcal F_m(B,C)-\mathcal R_m(B,C).
		\]
		Subtracting \eqref{master-R-recurrence} from
		\eqref{master-F-recurrence}, we obtain the homogeneous recurrence
		\begin{align}
			\mathcal E_m(B,C)
			&=\lambda_m(B,C)
			\mathcal E_{m+2}(Bq^2,Cq^2).
			\label{master-E-recurrence}
		\end{align}
		Applying \eqref{master-E-recurrence} successively gives, for every
		positive integer $N$,
		\begin{align}
			\mathcal E_m(B,C)
			&=\prod_{j=0}^{N-1}
			\lambda_{m+2j}(Bq^{2j},Cq^{2j})
			\mathcal E_{m+2N}(Bq^{2N},Cq^{2N}).
			\label{master-E-product}
		\end{align}
		By \eqref{master-lambda},
		\begin{align*}
			\lambda_{m+2j}(Bq^{2j},Cq^{2j})
			&=\frac{q^2(1-Bq^{2j})(1-Cq^{2j})}
			{(1+xq^{m+2j+2})(1+x^{-1}q^{m+2j+2})}.
		\end{align*}
		Consequently, \eqref{master-E-product} becomes
		\begin{align}
			\mathcal E_m(B,C)
			=\frac{q^{2N}(B,C;q^2)_N}
			{(-xq^{m+2},-x^{-1}q^{m+2};q^2)_N}
			\mathcal E_{m+2N}(Bq^{2N},Cq^{2N}).
			\label{master-E-iteration}
		\end{align}
		Since $|q|<1$, the numerator products in the quotient above remain
		bounded, while the denominator products converge to nonzero limits under
		the hypotheses of the lemma. Hence there is a constant $K>0$, independent
		of $N$, such that
		\[
		\left|
		\frac{(B,C;q^2)_N}
		{(-xq^{m+2},-x^{-1}q^{m+2};q^2)_N}
		\right|\leq K.
		\]
		Thus the coefficient in \eqref{master-E-iteration} is
		$O(q^{2N})$ and tends to zero. It remains to prove that the last factor
		in \eqref{master-E-iteration} is bounded uniformly in $N$.
		
		\smallskip
		\noindent\emph{Step 4: the boundary estimate and the cancellation.}
		The sole purpose of this step is to show that
		$\mathcal E_{m+2N}(Bq^{2N},Cq^{2N})$ remains bounded as $N\to\infty$.
		The two parts of $\mathcal R_{m+2N}$ need not be bounded separately;
		we shall compute only their leading terms and show that these terms cancel.
		Assume first that $C\neq0$ and retain the notation in
		\eqref{master-ctd}.  Put
		\begin{align*}
			S&=\sum_{n=0}^{\infty}
			\frac{(tq^2/C;q^2)_n(Cc)^n}{(Bc;q^2)_{n+1}}.
		\end{align*}
		For every fixed $a$, we have
		$(aq^{2N};q^2)_\infty=1+O(q^{2N})$. Hence the definition of
		$\mathcal U_m$ gives
		\begin{align}
			\mathcal U_{m+2N}(Bq^{2N},Cq^{2N})
			&=cq^{-2N}
			\frac{(Bq^{2N},Cq^{2N};q^2)_\infty}
			{(tq^{2N+2},q^{2N+2}/c;q^2)_\infty}S \notag\\
			&=cq^{-2N}S+O(1).
			\label{master-U-asymptotic}
		\end{align}
		We next show that $\mathcal V_{m+2N}$ has the opposite leading term.
		Applying Lemma~\ref{jackson-transformation} with
		$a=q^2$, $b=tq^2/C$, $e=Bcq^2$, and $z=Cc$, we obtain
		\begin{align}
			S=\sum_{n=0}^{\infty}
			\frac{(-1)^n(ct)^nq^{n(n+1)}(d;q^2)_n}
			{(Bc,Cc;q^2)_{n+1}}.
			\label{master-S-transform}
		\end{align}
		Here $b\neq0$ and $|z|=|Cxq^{-m}|<1$; the denominator conditions
		also follow from the hypotheses of Lemma~\ref{unified-identity}.
		
		On the other hand,
		\begin{align}
			\mathcal V_{m+2N}(Bq^{2N},Cq^{2N})
			=(1-tq^{2N})\sum_{n=0}^{\infty}
			\frac{(cq^{-2N};q^2)_{n+1}(d;q^2)_n(tq^{2N})^n}
			{(Bc,Cc;q^2)_{n+1}}.
			\label{master-V-limit}
		\end{align}
		Applying Lemma~\ref{finite-q-binomial} with $a=cq^{-2N}$ and
		$r=n+1$, and multiplying by $(tq^{2N})^n$, gives
		\begin{align}
			(cq^{-2N};q^2)_{n+1}(tq^{2N})^n
			=\sum_{j=0}^{n+1}(-1)^jc^jt^nq^{2N(n-j)}
			q^{j(j-1)}\genfrac{[}{]}{0pt}{}{n+1}{j}_{q^2}.
			\label{master-finite-binomial}
		\end{align}
		The term $j=n+1$ in \eqref{master-finite-binomial} is
		\begin{align*}
			-cq^{-2N}(-1)^n(ct)^nq^{n(n+1)}.
		\end{align*}
		By \eqref{master-S-transform}, its contribution to
		\eqref{master-V-limit} is
		\[
		-cq^{-2N}(1-tq^{2N})S=-cq^{-2N}S+O(1).
		\]
		All the other terms contribute only $O(1)$. The factors
		$(d;q^2)_n/(Bc,Cc;q^2)_{n+1}$ and
		$\genfrac{[}{]}{0pt}{}{n+1}{j}_{q^2}$ are uniformly bounded for
		$0\le j\le n$, as is $1-tq^{2N}$. Setting $n=j+k$ thus bounds
		each remaining summand in modulus by a constant independent of $N,j,k$
		times $|ct|^j|q|^{j(j-1)}|tq^{2N}|^k$.
		Since $|t|<1$,
		\[
		\sum_{j,k\ge0}|ct|^j|q|^{j(j-1)}|tq^{2N}|^k
		\le\frac{1}{1-|t|}\sum_{j\ge0}|ct|^j|q|^{j(j-1)}<\infty,
		\]
		uniformly in $N$; the last series converges by the ratio test.
		Therefore
		\begin{align}
			\mathcal V_{m+2N}(Bq^{2N},Cq^{2N})
			=-cq^{-2N}S+O(1).
			\label{master-V-asymptotic}
		\end{align}
		The two leading terms now cancel. By \eqref{master-U-asymptotic} and
		\eqref{master-V-asymptotic},
		\begin{align}
			\mathcal R_{m+2N}(Bq^{2N},Cq^{2N})
			=O(1).
			\label{master-R-bound}
		\end{align}
		The defining series shows that the corresponding $\mathcal F$-term is
		uniformly bounded. Together with \eqref{master-R-bound}, this implies that
		$\mathcal E_{m+2N}(Bq^{2N},Cq^{2N})$ is bounded uniformly in $N$; say
		\[
		\left|\mathcal E_{m+2N}(Bq^{2N},Cq^{2N})\right|\leq M
		\]
		for all sufficiently large $N$. Combining this estimate with
		\eqref{master-E-iteration} and the constant $K$ from Step~3 yields
		\[
		|\mathcal E_m(B,C)|\leq KM|q|^{2N}.
		\]
		The left-hand side is independent of $N$, whereas the right-hand side
		tends to zero. Therefore $\mathcal E_m(B,C)=0$, or equivalently
		$\mathcal F_m(B,C)=\mathcal R_m(B,C)$. This proves
		\eqref{master-identity} for $C\neq0$.
		
		\smallskip
		\noindent\emph{The case $C=0$.}
		Finally,
		\begin{align}
			\lim_{C\to0}(-xq^{m+2}/C;q^2)_n(-Cxq^{-m})^n
			=(-1)^nx^{2n}q^{n^2+n}.
			\label{master-C-zero}
		\end{align}
		The product equals $\prod_{r=0}^{n-1}(-Cxq^{-m}-x^2q^{2+2r})$,
		whose factors have modulus at most $1/2$ for all large $r$, uniformly
		for small $|C|$. Together with the denominator products being uniformly
		bounded away from zero and $|q|,|t|<1$, this gives uniform geometric
		majorants for all three series near $C=0$. The infinite-product prefactor
		is also continuous there, so passage to the limit proves
		\eqref{master-identity} for $C=0$.
	\end{proof}
	
	\section{Proofs of the main theorems}\label{sec-proofs}
	We first work in a nonempty open subset where all intermediate series
	converge absolutely. Since the final expressions are meromorphic in $x$,
	the identities extend to the stated domains, away from poles, by analytic
	continuation.

	\begin{proof}[Proof of Theorem~\ref{partial-thm1}]
		Taking $B=0$ and $C=0$ in Lemma~\ref{unified-identity}, and using \eqref{master-C-zero}, we have
		\begin{align}\label{partial-thm1-1}
			\sum^{\infty}_{n=0}\frac{q^{2n}}{(-xq^{m+2},-x^{-1}q^{m+2};q^2)_{n}}
			&=\frac{-xq^{-m}}{(-xq^{m+2},-x^{-1}q^{m+2};q^2)_{\infty}}
			\sum^{\infty}_{n=0}(-1)^nx^{2n}q^{n^2+n}\nonumber\\
			&\quad +(1+xq^m)\sum^{\infty}_{n=0}(-xq^{-m};q^2)_{n+1}(-xq^m)^n\nonumber\\
			&=\frac{-xq^{-m}}{(-xq^{m+2},-x^{-1}q^{m+2};q^2)_{\infty}}
			\sum^{\infty}_{n=0}(-1)^nx^{2n}q^{n^2+n}\nonumber\\
			&\quad +(1+xq^m)\left(\sum^{m-1}_{n=0}+\sum_{n=m}^{\infty}\right)(-xq^{-m};q^2)_{n+1}(-xq^m)^n.
		\end{align}
		For the tail, the change of index $n\mapsto n+m$ gives
		\begin{align}\label{partial-thm1-2}
			&\sum_{n=m}^{\infty}(-xq^{-m};q^2)_{n+1}(-xq^m)^n\nonumber\\*
			&=(-xq^{-m};q^2)_{m+1}(-xq^m)^{m}\sum_{n=0}^{\infty}(-xq^{m+2};q^2)_{n}(-xq^m)^{n}\nonumber\\
			&=(-xq^{-m};q^2)_{m}(-xq^m)^{m}\sum^{\infty}_{n=0}x^{3n}q^{3n^2+(3m+1)n}(1-x^2q^{4n+2m+2}),
		\end{align}
		where the last equality follows from \eqref{RF-zero-identity}
		with $\alpha=-xq^{m+2}$ and $\tau=-xq^m$.
		Substituting \eqref{partial-thm1-2} into
		\eqref{partial-thm1-1} and multiplying by $q^m$
		proves the result.
	\end{proof}
	
	\begin{proof}[Proof of Theorem~\ref{partial-thm2}]
		Taking $B=q$ and $C=0$ in Lemma~\ref{unified-identity}, and using \eqref{master-C-zero}, we deduce
		\begin{align}\label{partial-thm2-1}
			&\sum^{\infty}_{n=0}\frac{(q;q^2)_nq^{2n}}{(-xq^{m+2},-x^{-1}q^{m+2};q^2)_{n}} \nonumber\\
			&=\frac{-xq^{-m}(q;q^2)_{\infty}}{(-xq^{m+2},-x^{-1}q^{m+2};q^2)_{\infty}}\sum^{\infty}_{n=0}\frac{(-1)^nx^{2n}q^{n^2+n}}
			{(-xq^{-m+1};q^2)_{n+1}}\nonumber\\
			&\quad +(1+xq^m)\sum^{\infty}_{n=0}\frac{(-xq^{-m};q^2)_{n+1}(-xq^m)^n}{(-xq^{-m+1};q^2)_{n+1}} \nonumber\\
			&=\frac{-xq^{-m}(q;q^2)_{\infty}}{(-xq^{m+2},-x^{-1}q^{m+2};q^2)_{\infty}}\left(\sum^{m-1}_{n=0}
			+\sum^{\infty}_{n=m}\right)\frac{(-1)^nx^{2n}q^{n^2+n}}{(-xq^{-m+1};q^2)_{n+1}}\nonumber\\
			&\quad +(1+xq^m)\left(\sum^{m-1}_{n=0}
			+\sum^{\infty}_{n=m}\right)\frac{(-xq^{-m};q^2)_{n+1}(-xq^m)^n}{(-xq^{-m+1};q^2)_{n+1}}.
		\end{align}
		Observe that
		\begin{align}\label{partial-thm2-2}
			\sum^{\infty}_{n=m}\frac{(-1)^nx^{2n}q^{n^2+n}}{(-xq^{-m+1};q^2)_{n+1}}\nonumber
			&=\frac{(-1)^{m}x^{2m}q^{m^2+m}}{(-xq^{-m+1};q^2)_{m+1}}
			\sum^{\infty}_{n=0}\frac{(-1)^{n}x^{2n}q^{n^2+(2m+1)n}}{(-xq^{m+3};q^2)_n}\nonumber\\
			&=\frac{(-1)^{m}x^{2m}q^{m^2+m}}{(-xq^{-m+1};q^2)_{m}}\sum^{\infty}_{n=0}x^{3n}q^{3n^2+(3m+2)n}(1-xq^{2n+m+1}),
		\end{align}
		where the last step follows from \eqref{RF} with $q\mapsto q^2$,
		$\alpha=x^2q^{2m+2}/\tau$, $\beta=-xq^{m+3}$, and $\tau\to0$.
		Similarly,
		\begin{align}\label{partial-thm2-5}
			&\sum^{\infty}_{n=m}\frac{(-xq^{-m};q^2)_{n+1}(-xq^m)^n}{(-xq^{-m+1};q^2)_{n+1}}\nonumber\\
			&=\frac{(-xq^{-m};q^2)_{m+1}(-xq^m)^{m}}{(-xq^{-m+1};q^2)_{m+1}}
			\sum^{\infty}_{n=0}\frac{(-xq^{m+2};q^2)_{n}(-xq^m)^n}{(-xq^{m+3};q^2)_{n}}\nonumber\\
			&=\frac{(-xq^{-m};q^2)_{m}(-xq^m)^{m}}{(-xq^{-m+1};q^2)_{m}}\sum^{\infty}_{n=0}x^{2n}q^{2n^2+(2m+1)n}(1-xq^{2n+m+1}),
		\end{align}
		where the last equality follows from \eqref{RF} with $q\mapsto q^2$,
		$\alpha=-xq^{m+2}$, $\beta=-xq^{m+3}$, and $\tau=-xq^m$.
		Substituting \eqref{partial-thm2-2} and \eqref{partial-thm2-5} into
		\eqref{partial-thm2-1}, and multiplying by $q^m$, proves the result.
	\end{proof}

	\begin{proof}[Proof of Theorem~\ref{partial-thm3}]
		Taking $B=-1$ and $C=-q$ in Lemma~\ref{unified-identity}, we obtain
		\begin{equation}\label{thm3-initial}
			\sum^{\infty}_{n=0}\frac{(-1;q)_{2n}q^{2n}}{(-xq^{m+2},-x^{-1}q^{m+2};q^2)_{n}}=P_3+Q_3,
		\end{equation}
		where
		\begin{align}
			P_3&:=\frac{-xq^{-m}(-1;q)_{\infty}}{(-xq^{m+2},-x^{-1}q^{m+2};q^2)_{\infty}}
			\sum^{\infty}_{n=0}\frac{(xq^{m+1};q^2)_{n}(x q^{1-m})^n}{(xq^{-m};q^2)_{n+1}},\label{P3}\\
			Q_3&:=(1+xq^{m})(1+xq^{-m})\sum^{\infty}_{n=0}\frac{(q^{-2m+1},-xq^{-m+2};q^2)_{n}(-xq^m)^{n}}{(xq^{-m};q)_{2n+2}}.\nonumber
		\end{align}
		
		Set $q\mapsto q^{2}$, $\alpha=xq^{m+1}$,
		$\beta=xq^{-m+2}$, and $\tau=xq^{1-m}$ in \eqref{RF}.
		The remaining finite products simplify as follows:
		\[
		\frac{(xq^{m+1};q^2)_n(xq^{m+2};q^2)_n}
		{(1-xq^{-m})(xq^{-m+2};q^2)_n(xq^{1-m};q^2)_{n+1}}
		=\frac{(xq^{2n-m+2};q)_{2m-1}}{(xq^{-m};q)_{2m+1}}.
		\]
		The last equality follows by splitting the same product
		$(xq^{-m};q)_{2n+2m+1}$ in two ways.
		Consequently, for $|xq^{1-m}|<1$, we have
		\begin{align*}
			\sum^{\infty}_{n=0}\frac{(xq^{m+1};q^2)_{n}(x q^{1-m})^n}{(xq^{-m};q^2)_{n+1}}
			&=\frac{1}{(xq^{-m};q)_{2m+1}}\\
			&\quad \times\sum^{\infty}_{n=0}(xq^{2n-m+2};q)_{2m-1}x^{2n}q^{2n^2+(-2m+1)n}(1-x^2q^{4n+2}).
		\end{align*}
		Combining this identity with \eqref{P3} yields
		\begin{align}\label{P3-2}
			P_3&=\frac{-xq^{-m}(-1;q)_{\infty}}{(xq^{-m};q)_{2m+1}(-xq^{m+2},-x^{-1}q^{m+2};q^2)_{\infty}}\nonumber\\
			&\quad \times\sum^{\infty}_{n=0}(xq^{2n-m+2};q)_{2m-1}x^{2n}q^{2n^2+(-2m+1)n}(1-x^2q^{4n+2}).
		\end{align}
		To simplify $Q_3$, split its defining series at $n=m$:
		\begin{align*}
			Q_3={}&(1+xq^m)(1+xq^{-m})
			\sum_{n=0}^{m-1}
			\frac{(q^{-2m+1},-xq^{-m+2};q^2)_n(-xq^m)^n}
			{(xq^{-m};q)_{2n+2}}\\
			&+(1+xq^m)(1+xq^{-m})
			\sum_{n=m}^{\infty}
			\frac{(q^{-2m+1},-xq^{-m+2};q^2)_n(-xq^m)^n}
			{(xq^{-m};q)_{2n+2}}\\
			&=(1+xq^m)(1+xq^{-m})
			\sum_{n=0}^{m-1}
			\frac{(q^{-2m+1},-xq^{-m+2};q^2)_n(-xq^m)^n}
			{(xq^{-m};q)_{2n+2}}\\
			&+\frac{(1+xq^m)(1+xq^{-m})
				(q^{-2m+1},-xq^{-m+2};q^2)_m(-xq^m)^m}
			{(xq^{-m};q)_{2m+1}}\\*
			&\quad\times\sum_{n=0}^{\infty}
			\frac{(q,-xq^{m+2};q^2)_n(-xq^m)^n}
			{(xq^{m+1};q)_{2n+1}}.
		\end{align*}
		Applying \eqref{Heineh} to the last series with
		$(a,b,c,t,h)=(-xq^{m+2},q,xq^{m+2},-xq^m,2)$ gives
		\begin{align*}
			&\sum_{n=0}^{\infty}
			\frac{(q,-xq^{m+2};q^2)_n(-xq^m)^n}
			{(xq^{m+1};q)_{2n+1}}\\*
			&=\frac{(q;q)_\infty(x^2q^{2m+2};q^2)_\infty}
			{(1-xq^{m+1})(xq^{m+2};q)_\infty(-xq^m;q^2)_\infty}
			\sum_{n=0}^{\infty}
			\frac{(-xq^m;q^2)_nq^n}
			{(q;q)_n(-xq^{m+1};q)_n}.
		\end{align*}
		Then taking
		$(a,b,c,t,h)=(ix^{1/2}q^{m/2},-ix^{1/2}q^{m/2},
		-xq^{m+1},q,1)$ in  \eqref{Heineh} yields
		\begin{align*}
			&\sum_{n=0}^{\infty}
			\frac{(-xq^m;q^2)_nq^n}{(q;q)_n(-xq^{m+1};q)_n}\\
			&=\frac{(-ix^{1/2}q^{m/2};q)_\infty
				(ix^{1/2}q^{m/2+1};q)_\infty}
			{(-xq^{m+1};q)_\infty(q;q)_\infty}
			\sum_{n=0}^{\infty}
			\frac{(-ix^{1/2}q^{m/2+1};q)_n
				(-ix^{1/2}q^{m/2})^n}
			{(ix^{1/2}q^{m/2+1};q)_n}.
		\end{align*}
		The product of the two prefactors is
		\begin{align*}
			&\frac{(x^2q^{2m+2};q^2)_\infty
				(-ix^{1/2}q^{m/2};q)_\infty
				(ix^{1/2}q^{m/2+1};q)_\infty}
			{(1-xq^{m+1})(xq^{m+2};q)_\infty
				(-xq^m;q^2)_\infty(-xq^{m+1};q)_\infty}
			=\frac{1}{1-ix^{1/2}q^{m/2}}.
		\end{align*}
		Finally, substituting
		\[
		(\alpha,\beta,\tau)=(-ix^{1/2}q^{m/2+1},
		ix^{1/2}q^{m/2+1},-ix^{1/2}q^{m/2})
		\]
		into \eqref{RF} gives
		\begin{align*}
			\sum_{n=0}^{\infty}
			\frac{(-ix^{1/2}q^{m/2+1};q)_n
				(-ix^{1/2}q^{m/2})^n}
			{(ix^{1/2}q^{m/2+1};q)_n}=\frac{1}{1+ix^{1/2}q^{m/2}}
			\sum_{n=0}^{\infty}x^nq^{n^2+mn}(1+xq^{2n+m+1}).
		\end{align*}
		Consequently, the shifted tail is
		\[
		\sum_{n=0}^{\infty}
		\frac{(q,-xq^{m+2};q^2)_n(-xq^m)^n}
		{(xq^{m+1};q)_{2n+1}}
		=\frac{1}{1+xq^m}
		\sum_{n=0}^{\infty}x^nq^{n^2+mn}(1+xq^{2n+m+1}).
		\]
		Substituting this identity into the split series gives
		\begin{align}\label{Q3-4}
			Q_3&=(1+xq^{m})(1+xq^{-m})\sum^{m-1}_{n=0}\frac{(q^{-2m+1},-xq^{-m+2};q^2)_{n}(-xq^m)^{n}}{(xq^{-m};q)_{2n+2}}\nonumber \\
			&\quad +\frac{(1+xq^{-m})(-xq^{-m+2},q^{-2m+1};q^2)_{m}(-xq^m)^{m}}{(xq^{-m};q)_{2m+1}}
			\sum_{n=0}^{\infty}x^nq^{n^2+mn}(1+xq^{2n+m+1}).
		\end{align}
		
		Multiplying \eqref{thm3-initial} by
		$q^m/\bigl((1+xq^m)(1+x^{-1}q^m)\bigr)$
		and then substituting \eqref{P3-2} and \eqref{Q3-4} gives the asserted identity.
	\end{proof}

	\begin{proof}[Proof of Theorem~\ref{partial-thm4}]
		Taking $B=q$ and $C=-q$ in Lemma~\ref{unified-identity},
		we obtain
		\begin{equation}\label{thm4-initial}
			\sum^{\infty}_{n=0}\frac{(q^2;q^4)_{n}q^{2n}}{(-xq^{m+2},-x^{-1}q^{m+2};q^2)_{n}}=P_4+Q_4,
		\end{equation}
		where
		\begin{align}
			P_4&:=-xq^{-m}\frac{(q^2;q^4)_{\infty}}{(-xq^{m+2},-x^{-1}q^{m+2};q^2)_{\infty}}
			\sum^{\infty}_{n=0}\frac{(xq^{m+1};q^2)_{n}(xq^{1-m})^n}{(-xq^{1-m};q^2)_{n+1}},\label{P4}\\
			Q_4&:=(1+xq^{m})(1+xq^{-m})\sum^{\infty}_{n=0}
			\frac{(-xq^{-m+2},-q^{-2m+2};q^2)_{n}(-xq^m)^{n}}{(x^2q^{-2m+2};q^4)_{n+1}}.\nonumber
		\end{align}
		
		Replace $q$, $\alpha$, $\beta$, and $\tau$ by $q^2$, $xq^{m+1}$,
		$-xq^{3-m}$, and $xq^{1-m}$ in \eqref{RF}, respectively.
		The product quotient is
		\[
		\frac{(xq^{m+1};q^2)_n(-xq^{m+1};q^2)_n}
		{(1+xq^{1-m})(-xq^{3-m};q^2)_n(xq^{1-m};q^2)_{n+1}}
		=\frac{(x^2q^{4n-2m+6};q^4)_{m-1}}
		{(x^2q^{2-2m};q^4)_m}.
		\]
		This is the result of splitting
		$(x^2q^{2-2m};q^4)_{n+m}$ after either $m$ or $n+1$ factors.
		It follows, for $|xq^{1-m}|<1$, that
		\begin{align*}
			&\sum^{\infty}_{n=0}\frac{(xq^{m+1};q^2)_{n}(xq^{1-m})^n}{(-xq^{1-m};q^2)_{n+1}}\\
			&=\frac{1}{(x^2q^{2-2m};q^4)_m}
			\sum^{\infty}_{n=0}(-1)^n(x^2q^{4n-2m+6};q^4)_{m-1}x^{2n}q^{2n^2+(2-2m)n}(1-x^2q^{4n+2}).
		\end{align*}
		Substituting the above identity into \eqref{P4} yields
		\begin{align}\label{P4-2}
			P_4&=-xq^{-m}\frac{(q^2;q^4)_{\infty}}{(x^2q^{2-2m};q^4)_m(-xq^{m+2},-x^{-1}q^{m+2};q^2)_{\infty}} \nonumber\\
			&\quad \times \sum^{\infty}_{n=0}(-1)^n(x^2q^{4n-2m+6};q^4)_{m-1}x^{2n}q^{2n^2+(2-2m)n}(1-x^2q^{4n+2}).
		\end{align}
		
		To simplify $Q_4$, split its defining series at $n=m$:
		\begin{align}
			Q_4={}&(1+xq^m)(1+xq^{-m})
			\sum_{n=0}^{m-1}
			\frac{(-xq^{-m+2},-q^{-2m+2};q^2)_n(-xq^m)^n}
			{(x^2q^{-2m+2};q^4)_{n+1}}\nonumber\\
			&+(1+xq^m)(1+xq^{-m})
			\sum_{n=m}^{\infty}
			\frac{(-xq^{-m+2},-q^{-2m+2};q^2)_n(-xq^m)^n}
			{(x^2q^{-2m+2};q^4)_{n+1}}\nonumber\\
			={}&(1+xq^m)(1+xq^{-m})
			\sum_{n=0}^{m-1}
			\frac{(-xq^{-m+2},-q^{-2m+2};q^2)_n(-xq^m)^n}
			{(x^2q^{-2m+2};q^4)_{n+1}}\nonumber\\
			&+\frac{(1+xq^m)(1+xq^{-m})
				(-xq^{-m+2},-q^{-2m+2};q^2)_m(-xq^m)^m}
			{(x^2q^{-2m+2};q^4)_{m+1}}\nonumber\\
			&\quad\times\sum_{n=0}^{\infty}
			\frac{(-xq^{m+2},-q^2;q^2)_n(-xq^m)^n}
			{(x^2q^{2m+6};q^4)_n}.\label{q4-1}
		\end{align}
		Taking $q\mapsto q^2$ and
		$
		(a,b,c,d,e)=(q^2,-xq^{m+2},-q^2,-xq^{m+3},xq^{m+3})
		$
		in \eqref{3phi2} gives
		\begin{align}
			\sum_{n=0}^{\infty}
			\frac{(-xq^{m+2},-q^2;q^2)_n(-xq^m)^n}
			{(x^2q^{2m+6};q^4)_n}
			=\frac{1-xq^{m+1}}{1+xq^m}
			\sum_{n=0}^{\infty}
			\frac{(q,xq^{m+1};q^2)_n(xq^{m+1})^n}
			{(-xq^{m+2};q)_{2n}}.\label{q4-2}
		\end{align}
		Applying \eqref{Heineh} to the last series with
		$(a,b,c,t,h)=(xq^{m+1},q,-xq^{m+2},xq^{m+1},2)$ yields
		\begin{align}
			\sum_{n=0}^{\infty}
			\frac{(q,xq^{m+1};q^2)_n(xq^{m+1})^n}
			{(-xq^{m+2};q)_{2n}}
			=\frac{(q;q)_\infty(x^2q^{2m+2};q^2)_\infty}
			{(-xq^{m+2};q)_\infty(xq^{m+1};q^2)_\infty}
			\sum_{n=0}^{\infty}
			\frac{(xq^{m+1};q^2)_nq^n}
			{(q,xq^{m+1};q)_n}.\label{q4-3}
		\end{align}
		Then taking $(a,b,c,t,h)=(x^{1/2}q^{(m+1)/2},-x^{1/2}q^{(m+1)/2},
		xq^{m+1},q,1)$
		in \eqref{Heineh} gives
		\begin{align}
			&\sum_{n=0}^{\infty}
			\frac{(xq^{m+1};q^2)_nq^n}{(q,xq^{m+1};q)_n}\nonumber\\
			&=\frac{(-x^{1/2}q^{(m+1)/2},x^{1/2}q^{(m+3)/2};q)_\infty}
			{(xq^{m+1},q;q)_\infty}
			\sum_{n=0}^{\infty}
			\frac{(-x^{1/2}q^{(m+1)/2};q)_n
				(-x^{1/2}q^{(m+1)/2})^n}
			{(x^{1/2}q^{(m+3)/2};q)_n}.\label{q4-4}
		\end{align}
		Finally, substituting $
		(\alpha,\beta,\tau)=(-x^{1/2}q^{(m+1)/2},
		x^{1/2}q^{(m+3)/2},-x^{1/2}q^{(m+1)/2})$
		into \eqref{RF} gives
		\begin{align}
			\sum_{n=0}^{\infty}
			\frac{(-x^{1/2}q^{(m+1)/2};q)_n
				(-x^{1/2}q^{(m+1)/2})^n}
			{(x^{1/2}q^{(m+3)/2};q)_n}
			=(1-x^{1/2}q^{(m+1)/2})
			\sum_{n=0}^{\infty}(-1)^nx^nq^{n^2+(m+1)n}.\label{q4-5}
		\end{align}
		Consequently, combining \eqref{q4-2}--\eqref{q4-5}, we derive
		\begin{align}
			\sum_{n=0}^{\infty}
			\frac{(-xq^{m+2},-q^2;q^2)_n(-xq^m)^n}
			{(x^2q^{2m+6};q^4)_n}
			=\frac{1-x^2q^{2m+2}}{1+xq^m}
			\sum_{n=0}^{\infty}(-1)^nx^nq^{n^2+(m+1)n}.
			\label{Q4-tail}
		\end{align}
		Substitution of \eqref{Q4-tail} into \eqref{q4-1} gives
		\begin{align}\label{Q4-2}
			Q_4&=(1+xq^{m})(1+xq^{-m})\sum^{m-1}_{n=0}\frac{(-xq^{-m+2},-q^{-2m+2};q^2)_{n}(-xq^m)^{n}}{(x^2q^{2-2m};q^4)_{n+1}}
			\nonumber\\
			&\quad+(1+xq^{m})(1+xq^{-m})\frac{(-xq^{-m+2};q^2)_{m-1}(-q^{-2m+2};q^2)_{m}(-xq^m)^{m}}{(x^2q^{2-2m};q^4)_{m}}
			\nonumber\\
			&\quad\times \sum^{\infty}_{n=0}(-1)^nx^nq^{n^2+(m+1)n}.
		\end{align}
		
		Multiplying \eqref{thm4-initial} by
		$q^m/\bigl((1+xq^m)(1+x^{-1}q^m)\bigr)$
		and then substituting \eqref{P4-2} and \eqref{Q4-2} gives the asserted identity.
	\end{proof}

	\begin{proof}[Proof of Theorem~\ref{partial-thm5}]
		Taking $B=-q$ and $C=-q^2$ in Lemma~\ref{unified-identity} yields
		\begin{align}\label{thm5-initial}
			\sum^{\infty}_{n=0}\frac{(-q;q)_{2n}q^{2n}}{(-xq^{m+2},-x^{-1}q^{m+2};q^2)_{n}}=P_5+Q_5,
		\end{align}
		where
		\begin{align}
			P_5&:=\frac{-xq^{-m}(-q;q)_{\infty}}{(-xq^{m+2},-x^{-1}q^{m+2};q^2)_{\infty}}
			\sum^{\infty}_{n=0}\frac{(xq^m;q^2)_{n}(xq^{-m+2})^n}{(xq^{-m+1};q^2)_{n+1}},\label{P5}\\
			Q_5&:=(1+xq^{m})(1+xq^{-m})\sum^{\infty}_{n=0}
			\frac{(-xq^{-m+2},q^{-2m+3};q^2)_{n}(-xq^m)^n}{(xq^{-m+1};q)_{2n+2}}.\nonumber
		\end{align}
		
		Replace $q$, $\alpha$, $\beta$, and $\tau$ by $q^2$, $xq^m$,
		$xq^{-m+3}$, and $xq^{-m+2}$ in \eqref{RF}, respectively.
		The remaining product quotient is
		\[
		\frac{(xq^m;q^2)_n(xq^{m+1};q^2)_n}
		{(1-xq^{-m+1})(xq^{-m+3};q^2)_n(xq^{-m+2};q^2)_{n+1}}
		=\frac{(xq^{2n-m+3};q)_{2m-3}}{(xq^{-m+1};q)_{2m-1}}.
		\]
		Here both sides follow by splitting
		$(xq^{-m+1};q)_{2n+2m-1}$ in two ways.
		Therefore, for $|xq^{2-m}|<1$,
		\begin{align*}
			\sum^{\infty}_{n=0}\frac{(xq^m;q^2)_{n}(xq^{-m+2})^n}{(xq^{-m+1};q^2)_{n+1}}&=\frac{1}{(xq^{-m+1};q)_{2m-1}}\\
			&\quad\times \sum^{\infty}_{n=0}(xq^{2n-m+3};q)_{2m-3}x^{2n}q^{2n^2+(3-2m)n}(1-x^2q^{4n+2}).
		\end{align*}
		Combining this identity with \eqref{P5} yields
		\begin{align}\label{P5-2}
			P_5&=-\frac{xq^{-m}(-q;q)_{\infty}}{(xq^{-m+1};q)_{2m-1}(-xq^{m+2},-x^{-1}q^{m+2};q^2)_{\infty}} \nonumber \\
			&\quad \times \sum^{\infty}_{n=0}(xq^{2n-m+3};q)_{2m-3}x^{2n}q^{2n^2+(3-2m)n}(1-x^2q^{4n+2}).
		\end{align}
		
		To simplify $Q_5$, split its defining series at $n=m-1$:
		\begin{align}
			Q_5={}&(1+xq^m)(1+xq^{-m})
			\sum_{n=0}^{m-2}
			\frac{(-xq^{-m+2},q^{-2m+3};q^2)_n(-xq^m)^n}
			{(xq^{-m+1};q)_{2n+2}}\nonumber\\
			&+(1+xq^m)(1+xq^{-m})
			\sum_{n=m-1}^{\infty}
			\frac{(-xq^{-m+2},q^{-2m+3};q^2)_n(-xq^m)^n}
			{(xq^{-m+1};q)_{2n+2}}\nonumber\\
			={}&(1+xq^m)(1+xq^{-m})
			\sum_{n=0}^{m-2}
			\frac{(-xq^{-m+2},q^{-2m+3};q^2)_n(-xq^m)^n}
			{(xq^{-m+1};q)_{2n+2}}\nonumber\\
			&+\frac{(1+xq^m)(1+xq^{-m})
				(-xq^{-m+2},q^{-2m+3};q^2)_{m-1}(-xq^m)^{m-1}}
			{(xq^{-m+1};q)_{2m}}\nonumber\\
			&\quad\times\sum_{n=0}^{\infty}
			\frac{(-xq^m,q;q^2)_n(-xq^m)^n}
			{(xq^{m+1};q)_{2n}}.\label{q5-1}
		\end{align}
		Applying \eqref{Heineh} to the last series with
		$(a,b,c,t,h)=(-xq^m,q,xq^{m+1},-xq^m,2)$ gives
		\begin{align}
			\sum_{n=0}^{\infty}
			\frac{(-xq^m,q;q^2)_n(-xq^m)^n}
			{(xq^{m+1};q)_{2n}}=\frac{(1-xq^m)(-xq^m,q;q)_\infty}
			{(-xq^m;q^2)_\infty}
			\sum_{n=0}^{\infty}
			\frac{(-xq^m;q^2)_nq^n}{(-xq^m,q;q)_n}.\label{q5-2}
		\end{align}
		Then taking $(a,b,c,t,h)=(ix^{1/2}q^{m/2},-ix^{1/2}q^{m/2},
		-xq^m,q,1)$
		in \eqref{Heineh} yields
		\begin{align}
			&\sum_{n=0}^{\infty}
			\frac{(-xq^m;q^2)_nq^n}{(-xq^m,q;q)_n}\nonumber\\*
			&=\frac{(ix^{1/2}q^{m/2+1},-ix^{1/2}q^{m/2};q)_\infty}
			{(-xq^m,q;q)_\infty}
			\sum_{n=0}^{\infty}
			\frac{(-ix^{1/2}q^{m/2};q)_n
				(-ix^{1/2}q^{m/2})^n}
			{(ix^{1/2}q^{m/2+1};q)_n}.\label{q5-3}
		\end{align}
		Finally, substituting $(\alpha,\beta,\tau)=(-ix^{1/2}q^{m/2},
		ix^{1/2}q^{m/2+1},-ix^{1/2}q^{m/2})$
		into \eqref{RF} gives
		\begin{align}
			\sum_{n=0}^{\infty}
			\frac{(-ix^{1/2}q^{m/2};q)_n
				(-ix^{1/2}q^{m/2})^n}
			{(ix^{1/2}q^{m/2+1};q)_n}
			=(1-ix^{1/2}q^{m/2})
			\sum_{n=0}^{\infty}x^nq^{n^2+mn}.\label{q5-4}
		\end{align}
		Consequently, combining \eqref{q5-2}--\eqref{q5-4}, we derive
		\begin{align}
			\sum_{n=0}^{\infty}
			\frac{(-xq^m,q;q^2)_n(-xq^m)^n}
			{(xq^{m+1};q)_{2n}}
			=(1-xq^m)\sum_{n=0}^{\infty}x^nq^{n^2+mn}.\label{Q5-tail}
		\end{align}
		Substitution of \eqref{Q5-tail} into \eqref{q5-1} gives
		\begin{align}\label{Q5-final}
			Q_5&=(1+xq^{m})(1+xq^{-m})\sum^{m-2}_{n=0}\frac{(-xq^{-m+2},q^{-2m+3};q^2)_n(-xq^m)^n}{(xq^{-m+1};q)_{2n+2}}\nonumber \\*
			&\quad+\frac{(1+xq^{m})(1+xq^{-m})(-xq^{-m+2},q^{-2m+3};q^2)_{m-1}(-xq^m)^{m-1}}{(xq^{-m+1};q)_{2m-1}}\sum^{\infty}_{n=0}x^nq^{n^2+mn}.
		\end{align}
		
		Multiplying \eqref{thm5-initial} by
		$q^m/\bigl((1+xq^m)(1+x^{-1}q^m)\bigr)$
		and then substituting \eqref{P5-2} and \eqref{Q5-final} gives the asserted identity.
	\end{proof}
	
	\section{Special cases}\label{sec-special-cases}
	
	In this section, we show that the six Ramanujan identities displayed in the introduction are special cases of the main theorems.
	
	For completeness, we give the details of the first specialization.  When
	$m=0$, Theorem~\ref{partial-thm1} becomes
	\begin{align}
		\sum_{n=0}^{\infty}
		\frac{q^{2n}}{(-xq^2,-x^{-1}q^2;q^2)_n}
		&=(1+x)\sum_{n=0}^{\infty}
		x^{3n}q^{3n^2+n}(1-x^2q^{4n+2})
		\nonumber\\
		&\quad+\frac{1}{(-xq^2,-x^{-1}q^2;q^2)_\infty}
		\sum_{n=0}^{\infty}(-1)^{n+1}x^{2n+1}q^{n^2+n}.
		\label{special-p1}
	\end{align}
	Replacing $x$ by $-x$ in \eqref{special-p1}, and then replacing $q$ by
	$q^{1/2}$, gives \eqref{p-1}:
	\begin{align*}
		\sum_{n=0}^{\infty}\frac{q^n}{(xq,x^{-1}q;q)_n}
		&=(1-x)\sum_{n=0}^{\infty}(-1)^nx^{3n}
		q^{n(3n+1)/2}(1-x^2q^{2n+1})\\*
		&\quad+\frac{x}{(xq,x^{-1}q;q)_\infty}
		\sum_{n=0}^{\infty}(-1)^nx^{2n}q^{\binom{n+1}{2}}.
	\end{align*}
	
	Next, Theorem~\ref{partial-thm1} with $m=1$ gives
	\begin{align*}
		&\sum^{\infty}_{n=0}\frac{q^{2n+1}}{(-xq^3,-x^{-1}q^3;q^2)_{n}}-q(1+xq)(1+xq^{-1})\\
		&=-q^2(1+xq)(1+xq^{-1})\sum_{n=0}^{\infty}x^{3n+1}q^{n(3n+4)}(1-x^2q^{4n+4})\nonumber\\
		&\quad -\frac{1}{(-xq^3,-x^{-1}q^3;q^2)_{\infty}}\sum_{n=0}^{\infty}(-1)^nx^{2n+1}q^{n(n+1)}.
	\end{align*}
	Dividing both sides of the above identity by$(1+xq)(1+x^{-1}q)$,
	we have
	\begin{align*}
		&\sum^{\infty}_{n=0}\frac{q^{2n+1}}{(-xq,-x^{-1}q;q^2)_{n+1}}\\
		&=x-xq\sum_{n=0}^{\infty}x^{3n+1}q^{n(3n+4)}(1-x^2q^{4n+4})-\frac{1}{(-xq,-x^{-1}q;q^2)_{\infty}}\sum_{n=0}^{\infty}(-1)^nx^{2n+1}q^{n(n+1)}\\
		&=x-\sum_{n=0}^{\infty}x^{3n+2}q^{n(3n+4)+1}+\sum_{n=0}^{\infty}x^{3n+4}q^{n(3n+8)+5}\nonumber\\
		&\quad -\frac{1}{(-xq,-x^{-1}q;q^2)_{\infty}}\sum_{n=0}^{\infty}(-1)^nx^{2n+1}q^{n(n+1)}\\
		&=x-\sum_{n=0}^{\infty}x^{3n+2}q^{n(3n+4)+1}+\sum_{n=1}^{\infty}x^{3n+1}q^{n(3n+2)}\nonumber\\
		&\quad -\frac{1}{(-xq,-x^{-1}q;q^2)_{\infty}}\sum_{n=0}^{\infty}(-1)^nx^{2n+1}q^{n(n+1)}
	\end{align*}
	which is \eqref{p-2} after combining the three partial theta sums.
	
	Theorem~\ref{partial-thm2} with $m=0$ gives \eqref{p-6}, while its case $m=1$ gives \eqref{p-3}. Setting $m=0$ and then replacing $q$ by $q^{1/2}$ in Theorem~\ref{partial-thm4} gives \eqref{p-4}. Finally, setting $m=1$ and replacing $x$ by $-x$ in Theorem~\ref{partial-thm5} gives \eqref{p-5}.

	\section{A conjugate Bailey pair}\label{sec-bailey}

	\begingroup
	\interlinepenalty=10000
	Following \cite[(1.2)--(1.4)]{Lovejoy-2012}, we use the following
	specialization of the standard definitions. For fixed $m\geq1$, two
	sequences $(\alpha_n,\beta_n)$ form a Bailey pair relative to
	$a^2q^{2m}$, with base $q^2$, if
	\[
		\beta_n=\sum_{r=0}^{n}
		\frac{\alpha_r}
		{(q^2;q^2)_{n-r}(a^2q^{2m+2};q^2)_{n+r}}.
	\]
	\par\endgroup
	Two sequences $(\delta_n,\gamma_n)$ form a conjugate Bailey pair relative
	to $a^2q^{2m}$, with base $q^2$, if
	\[
		\gamma_r=\sum_{n=r}^{\infty}
		\frac{\delta_n}
		{(q^2;q^2)_{n-r}(a^2q^{2m+2};q^2)_{n+r}}.
	\]
	Under absolute convergence, the Bailey transform then gives
	\begin{equation}\label{bailey-transform}
		\sum_{n=0}^{\infty}\beta_n\delta_n
		=\sum_{n=0}^{\infty}\alpha_n\gamma_n.
	\end{equation}

	We first extract the residual identity needed below.

	\begin{thm}\label{residual-identity}
		For $m\geq1$,
		\begin{align}
			&\sum_{n=0}^{\infty}
			\frac{(aq;q)_{2n}q^{2n}}
			{(a^2q^{2m+2},q^2;q^2)_n}\nonumber\\
			&\quad=\frac{(aq^{2m};q)_{\infty}}
			{(a^2q^{2m+2},q^2;q^2)_{\infty}}
			\sum_{n=0}^{\infty}(aq^{2n+3};q)_{2m-3}a^{2n}
			q^{2n^2+3n}(1-a^2q^{2m+4n+2}).
			\label{m-residual}
		\end{align}
		The identity holds whenever the displayed denominators do not vanish.
	\end{thm}

	\begin{proof}
		In Theorem~\ref{partial-thm5}, replace $x$ by $zq^m$. Fix $N\geq0$,
		multiply the resulting identity by $1+z^{-1}q^{2N}$, and let
		$z\to -q^{2N}$. After the pole is removed, the series converge locally
		uniformly near this point, so the limit may be taken termwise.
		The finite correction terms are regular at this point
		and hence tend to zero. On the left-hand side, the terms with $n<N$
		also tend to zero. In the remaining terms write $n=N+k$. The factor
		that produces the pole satisfies
		\begin{align}
			\lim_{z\to-q^{2N}}
			\frac{(-z^{-1};q^2)_{N+k+1}}
			{1+z^{-1}q^{2N}}=(q^{-2N};q^2)_N(q^2;q^2)_k
			=(-1)^Nq^{-N(N+1)}(q^2;q^2)_N(q^2;q^2)_k.
			\label{removed-pole}
		\end{align}
		Consequently, the limit of the left-hand side is
		\begin{align}
			C_N\sum_{k=0}^{\infty}
			\frac{(-q^{2N+1};q)_{2k}q^{2k}}
			{(q^{4N+2m+2},q^2;q^2)_k},
			\label{residue-left}
		\end{align}
		where
		\[
			C_N=\frac{(-1)^Nq^{N(N+1)}(-q;q)_{2N}q^{2N+m}}
			{(q^2;q^2)_N(q^{2m+2N};q^2)_{N+1}}.
		\]
		Only the first infinite-series term on the right-hand side of
		Theorem~\ref{partial-thm5} has a pole. Using \eqref{removed-pole} and
		\[
			\frac{(-q;q)_\infty}
			{(-q;q)_{2N}(-q^{2N+1};q)_{2m-1}}
			=(-q^{2N+2m};q)_\infty,
		\]
		its limit simplifies to
		\begin{align}
			C_N\frac{(-q^{2N+2m};q)_\infty}
			{(q^{4N+2m+2},q^2;q^2)_\infty}\sum_{n=0}^{\infty}
			(-q^{2N+2n+3};q)_{2m-3}q^{4Nn}
			q^{2n^2+3n}(1-q^{4N+2m+4n+2}).
			\label{residue-right}
		\end{align}
		After canceling $C_N$, equations \eqref{residue-left} and
		\eqref{residue-right} give \eqref{m-residual} at $a=-q^{2N}$.
		Both sides of \eqref{m-residual} are analytic in a neighborhood of
		$a=0$. Since the points $-q^{2N}$ accumulate at zero, the identity
		theorem proves \eqref{m-residual} near zero, and meromorphic
		continuation proves it wherever the displayed expressions are defined.
	\end{proof}

	\begin{cor}\label{conjugate-bailey-pair}
		Let $m\geq1$. Then
		\[
			\delta_n=(aq;q)_{2n}q^{2n}
		\]
		and
		\begin{align}
			\gamma_r
			&=\frac{q^{2r}(aq;q)_{2r}(aq^{2m};q)_\infty}
			{(aq^{2m};q)_{2r}(a^2q^{2m+2},q^2;q^2)_\infty}
			\nonumber\\
			&\quad\times\sum_{k=0}^{\infty}
			(aq^{2r+2k+3};q)_{2m-3}a^{2k}
			q^{2k^2+(4r+3)k}
			(1-a^2q^{4r+2m+4k+2})
			\label{gamma-pair}
		\end{align}
		form a conjugate Bailey pair relative to $a^2q^{2m}$, with base $q^2$.
	\end{cor}

	\begin{proof}
		By the definition of a conjugate Bailey pair and the change of index
		$n=r+k$,
		\begin{align*}
			\gamma_r
			&=\sum_{n=r}^{\infty}
			\frac{(aq;q)_{2n}q^{2n}}
			{(a^2q^{2m+2};q^2)_{n+r}(q^2;q^2)_{n-r}}\\
			&=\frac{q^{2r}(aq;q)_{2r}}
			{(a^2q^{2m+2};q^2)_{2r}}
			\sum_{k=0}^{\infty}
			\frac{(aq^{2r+1};q)_{2k}q^{2k}}
			{(a^2q^{4r+2m+2},q^2;q^2)_k}.
		\end{align*}
		Apply Theorem~\ref{residual-identity} to the last series with $a$
		replaced by $aq^{2r}$. Splitting the finite and infinite products at
		$2r$ gives exactly \eqref{gamma-pair}.
	\end{proof}

	Thus, if $(\alpha_n,\beta_n)$ is any Bailey pair relative to
	$a^2q^{2m}$ with base $q^2$, equation \eqref{bailey-transform} gives,
	provided the interchange of sums is justified,
	\begin{equation*}
		\sum_{n=0}^{\infty}(aq;q)_{2n}q^{2n}\beta_n
		=\sum_{n=0}^{\infty}\alpha_n\gamma_n,
	\end{equation*}
	where $\gamma_n$ is given by \eqref{gamma-pair}.

	\begin{rem}
		When $m=1$,
		\[
			(aq^{2n+3};q)_{-1}(1-a^2q^{4n+4})
			=1+aq^{2n+2}.
		\]
		Splitting a single series into its even and odd parts therefore gives
		\[
			\sum_{n=0}^{\infty}a^{2n}q^{2n^2+3n}
			(1+aq^{2n+2})
			=\sum_{n=0}^{\infty}a^nq^{\binom{n+1}{2}+n}.
		\]
		Hence \eqref{m-residual} reduces to the residual identity
		\cite[Eq.~(2.7)]{Lovejoy-2012}; Corollary~\ref{conjugate-bailey-pair}
		is its $m$-parameter extension at the level of conjugate Bailey pairs.
	\end{rem}


\end{document}